\documentclass[12pt]{amsart}
\usepackage{amsfonts,amssymb,amscd,amstext,mathrsfs,color}

\input xy
\xyoption{all}

\newtheorem{theorem}{Theorem}

\newtheorem{lemma}{Lemma}
\newtheorem{proposition}{Proposition}
\newtheorem{claim}{Claim}

\theoremstyle{definition}

\newcommand{\C}{\mathbb{C}}
\renewcommand{\O}{\mathscr{O}}
\newcommand{\Aut}{{\operatorname{Aut}}}
\newcommand{\id}{\mathrm{id}}
\renewcommand{\Im}{{\operatorname{Im}}}

\begin{document}

\title{Chaotic holomorphic automorphisms \\ of Stein manifolds \\ with the volume density property}

\author{Leandro Arosio and Finnur L\'arusson}

\address{Dipartimento Di Matematica, Universit\`a di Roma \lq\lq Tor Vergata\rq\rq, Via Della Ricerca Scientifica 1, 00133 Roma, Italy}
\email{arosio@mat.uniroma2.it}

\address{School of Mathematical Sciences, University of Adelaide, Adelaide SA 5005, Australia}
\email{finnur.larusson@adelaide.edu.au}

\thanks{L.~Arosio was supported by SIR grant \lq\lq NEWHOLITE -- New methods in holomorphic iteration\rq\rq, no.~RBSI14CFME.  F.~L\'arusson was supported by Australian Research Council grant DP150103442.}


\date{3 May 2018.  Minor edits 29 June 2018}


\begin{abstract}
Let $X$ be a Stein manifold of dimension $n\geq 2$ satisfying the volume density property with respect to an exact holomorphic volume form.  For example, $X$ could be $\C^n$, any connected linear algebraic group that is not reductive, the Koras-Russell cubic, or a product $Y\times\C$, where $Y$ is any Stein manifold with the volume density property.  

We prove that chaotic automorphisms are generic among volume-preserving holomorphic automorphisms of $X$.  In particular, $X$ has a chaotic holomorphic automorphism.  A proof for $X=\C^n$ may be found in work of Forn\ae ss and Sibony.  We follow their approach closely.  

Peters, Vivas, and Wold showed that a generic volume-preserving automorphism of $\C^n$, $n\geq 2$, has a hyperbolic fixed point whose stable manifold is dense in $\C^n$.  This property can be interpreted as a kind of chaos.  We generalise their theorem to a Stein manifold as above.
\end{abstract}

\maketitle

\section{Introduction} 
\label{sec:intro}

\subsection{}
Complex manifolds can be thought of as laid out across a spectrum characterised by rigidity at one end and flexibility at the other.  On the rigid side, Kobayashi-hyperbolic manifolds have at most a finite-dimensional group of symmetries.  On the flexible side, there are manifolds with an extremely large group of holomorphic automorphisms, the prototypes being the affine spaces $\C^n$ for $n\geq 2$.  From a dynamical point of view, hyperbolicity does not permit chaos.  An endomorphism of a Kobayashi-hyperbolic manifold is non-expansive with respect to the Kobayashi distance, so every family of endomorphisms is equicontinuous.  In this paper, we show that not only does flexibility allow chaos: under a strong anti-hyperbolicity assumption, chaotic automorphisms are generic.  Our main result is the following.

\begin{theorem}  \label{t:main}
Let $X$ be a Stein manifold of dimension at least $2$ satisfying the volume density property with respect to an exact holomorphic volume form.  Chaotic automorphisms are generic among volume-preserving holomorphic automorphisms of $X$.  In particular, $X$ has a chaotic holomorphic automorphism.
\end{theorem}

The assumption that the volume form is exact obviously holds if $H^n(X)\footnote{All cohomology in this paper is complex de Rham cohomology.}=0$, where $n=\dim X$.  Recall that $X$ has the homotopy type of a smooth manifold $S$ of real dimension at most $n$.  If $S$ is noncompact or $S$ is not orientable or $\dim S<n$, then $H^n(X)\cong H^n(S)=0$.  Hence, the volume form can fail to be exact only if $S$ is compact and orientable and $\dim S=n$; then $H^n(X)$ is 1-dimensional.

Let us explain some of the terms in the theorem.  We adopt the popular definition of chaos due to Devaney \cite{Devaney1989} and simplified by Banks et al.~\cite{BBCDS1992}: an endomorphism is chaotic if it has a dense forward orbit and its periodic points are dense.  An equivalent and, for our purposes, more convenient formulation was given by Touhey \cite{Touhey1997}: an endomorphism is chaotic if and only if for every pair of nonempty open subsets, there is a cycle that visits both of them.

A holomorphic volume form $\omega$ is a nowhere-vanishing holomorphic section of the canonical bundle.  So a Stein manifold $X$ as in the theorem has a trivial canonical bundle (holomorphically trivial or, equivalently by Grauert's Oka principle, topologically trivial).  To say that an automorphism $f$ of $X$ is volume-preserving with respect to $\omega$ means that $f^*\omega=\omega$.  Then $f$ also preserves the real volume form $\omega\wedge\bar\omega$.

The group $\Aut(X)$ of holomorphic automorphisms of $X$ is a topological group with respect to the compact-open topology, which is defined by a complete metric.  The volume-preserving automorphisms form a closed subgroup $\Aut_\omega(X)$.  Theorem \ref{t:main} states that $\Aut_\omega(X)$ contains a dense $G_\delta$ subset consisting of chaotic automorphisms.

The volume density property of $X$ with respect to $\omega$ was introduced by Varolin \cite{Varolin2001} (see also \cite[Section 4.10]{Forstneric2017}).  It means that in the Lie algebra $\mathfrak g$ of holomorphic vector fields $v$ on $X$ whose $\omega$-divergence vanishes (meaning that the Lie derivative of $\omega$ along $v$ is zero or, equivalently, that the contraction $v\rfloor\omega$ is a closed form), the complete fields generate a dense subalgebra.  Here, $v$ is called complete if its flow exists for all complex time; the $\omega$-divergence of $v$ vanishes if and only if the time maps of its flow preserve $\omega$.  For a Stein manifold of dimension at least $2$, the Lie algebra $\mathfrak g$ is infinite-dimensional.

If a Stein manifold $X$ has the volume density property, then its tangent bundle is spanned by finitely many complete holomorphic vector fields of divergence zero.  Hence $X$ admits a dominating spray, so it is elliptic in the sense of Gromov and therefore an Oka manifold.  Also, if $\dim X\geq 2$, then the group $\Aut_\omega(X)$ acts infinitely transitively on $X$, that is, $k$-transitively for every $k\geq 1$.  For these and several other properties of Stein manifolds with the volume density property, see \cite[Section 2]{KK2011}.

\subsection{}
The first example of a Stein manifold with the volume density property (besides the trivial examples $\C$ with $dz$ and $\C^*$ with $dz/z$) is $\C^n$, $n\geq 2$, with the standard holomorphic volume form $\omega_0=dz_1\wedge\cdots\wedge dz_n$; this is due to Anders\'en \cite{Andersen1990}.  In this important special case, Theorem \ref{t:main} is not new.  It was proved, although not explicitly stated, by Forn\ae ss and Sibony in their groundbreaking paper~\cite{FS1997}.

\begin{theorem}[Forn\ae ss and Sibony, 1997]  \label{t:FS}
A generic volume-preserving holomorphic automorphism of $\C^n$, $n\geq 2$, is chaotic.
\end{theorem}

In particular, $\C^n$, $n\geq 2$, has a chaotic automorphism.  We do not know an explicit example.  By conjugating a chaotic volume-preserving automorphism, we see that $\C^n$ has chaotic automorphisms that do not preserve volume.  It is an interesting open problem to describe the closure of the set of chaotic automorphisms in $\Aut(\C^n)$.

In Section 2 of \cite{FS1997}, Forn\ae ss and Sibony prove that a generic biholomorphic symplectomorphism of $\C^{2k}$, $k\geq 1$, is chaotic (their Remark 2.5 and Proposition 2.7).  A similar proof works for volume-preserving automorphisms of $\C^n$, $n\geq 2$.  Let us call an automorphism \textit{expelling} if the $G_\delta$ set of points with unbounded forward orbit is dense.  Clearly, an automorphism with a dense forward orbit is expelling, so a chaotic automorphism is expelling.  The main step in Forn\ae ss and Sibony's proof of Theorem \ref{t:FS} is their Theorem 3.1, which says that a generic volume-preserving automorphism of $\C^n$ is expelling.  In their Remark 3.5, they point out that the proof of their Lemma 2.3 in the symplectic case can be adapted to the volume-preserving case.  The proof of Theorem \ref{t:FS} then goes through as in the symplectic case.

Peters, Vivas, and Wold showed that a generic volume-preserving automorphism of $\C^n$, $n\geq 2$, has a hyperbolic fixed point whose stable manifold is dense in $\C^n$ \cite[Theorem 1.1]{PVW2008}.  This property can be interpreted as a kind of chaos.  Their proof relies on \cite[Theorem 3.1]{FS1997} and may be generalised to the setting of Theorem \ref{t:main}.  (The proof of \cite[Theorem 1.1]{PVW2008} is not entirely correct.  We thank the authors for kindly showing us how to correct it.)

\begin{theorem}  \label{t:dense-stable-manifold}
Let $X$ be a Stein manifold of dimension at least $2$ satisfying the volume density property with respect to an exact holomorphic volume form.  A generic volume-preserving holomorphic automorphism of $X$ admits a hyperbolic fixed point whose stable manifold is dense in $X$.
\end{theorem}

Peters, Vivas, and Wold point out that polynomial automorphisms with an attracting fixed point at infinity are dense in $\Aut_{\omega_0}(\C^n)$ \cite[Remark 3.1]{PVW2008}.  It follows that the residual subset of chaotic automorphisms has empty interior in $\Aut_{\omega_0}(\C^n)$.  In \cite[Remark 3.1]{PVW2008}, they also point out that since a polynomial automorphism of $\C^2$ either has trivial dynamics or is conjugate by a polynomial automorphism to a composition of H\'enon maps, in which case it has an attracting fixed point at infinity, it cannot have a dense attracting set.  For the same reason, a chaotic automorphism of $\C^2$ cannot be polynomial.

\subsection{}
Let us recall some results and examples from Anders\'en-Lempert theory in order to clarify the scope of Theorem \ref{t:main}.  For more details, see \cite[Section 4.10]{Forstneric2017} and \cite{KK2011}.

Most known examples of Stein manifolds with the volume density property are affine algebraic manifolds with the algebraic volume density property.  (The first non-algebraic examples were found only very recently by Ramos-Peon \cite{Ramos-Peon2017}.)  An affine algebraic manifold $X$ with an algebraic volume form $\omega$ is said to have the algebraic volume density property with respect to $\omega$ if the complete algebraic vector fields with vanishing $\omega$-divergence generate the Lie algebra of algebraic vector fields on $X$ with vanishing $\omega$-divergence.  (The time maps of a complete algebraic vector field are in general not algebraic, only holomorphic.)  The algebraic volume density property implies the volume density property.  

The known ways to obtain new manifolds with the volume density property or the algebraic volume density property from old are very limited.  The product of two affine algebraic manifolds with the algebraic volume density property has the algebraic volume density property (with respect to the product volume form), and the product of two Stein manifolds with the volume density property has the volume density property.  The former result is not trivial and the latter is a deep theorem of Kaliman and Kutzschebauch; see \cite[Section 4]{KK2011}.  It follows that if $Y$ is a Stein manifold with the volume density property, then $X=Y\times\C$ is covered by Theorem \ref{t:main}. 

Most known examples of manifolds with the algebraic volume density property are captured by the following theorem of Kaliman and Kutzschebauch \cite[Theorem 1.3]{KK2017}.

\begin{theorem}[Kaliman and Kutzschebauch, 2017]  \label{t:affine}
Let $X$ be an affine homogeneous space of a linear algebraic group $G$.  Suppose that $X$ has a $G$-invariant algebraic volume form $\omega$.  Then $X$ has the algebraic volume density property with respect to $\omega$.
\end{theorem}

If $\omega$ is exact, say if $H^n(X)=0$, where $n=\dim X\geq 2$, then $X$ is captured by Theorem \ref{t:main}.

By Theorem \ref{t:affine}, every connected linear algebraic group has the algebraic volume density property with respect to a left- or right-invariant Haar form (this was first proved as \cite[Theorem 2]{KK2010}).  For example, $(\C^*)^n$ has the algebraic volume density property with respect to the volume form $(z_1\cdots z_n)^{-1}dz_1\wedge\cdots\wedge dz_n$.  Every connected linear algebraic group is the semidirect product of its unipotent radical and a maximal reductive subgroup.  The exponential map of a unipotent group is a biholomorphism.  It follows that a connected non-reductive linear algebraic group $G$ of dimension $n\geq 2$ has $H^n(G)=0$ and is covered by Theorem \ref{t:main}.  On the other hand, a bi-invariant Haar form on a reductive group (such as $(z_1\cdots z_n)^{-1}dz_1\wedge\cdots\wedge dz_n$ on $(\C^*)^n$) is not exact \cite[Proposition 8.6]{KK2016}, so Theorem \ref{t:main} does not capture any reductive groups.

Leuenberger \cite{Leuenberger2016} produced new examples of algebraic hypersurfaces in affine space with the volume density property, including the famous Koras-Russell cubic
\[ C=\{ (x,y,z,w)\in\C^4 : x^2y+x+z^2+w^3=0 \} \]
with the volume form $x^{-2}dx\wedge dz\wedge dw$.  It is known that $C$ is diffeomorphic to $\mathbb R^6$ but not algebraically isomorphic to $\C^3$ (in fact, the algebraic automorphism group does not act transitively on $C$).  Whether $C$ is biholomorphic to $\C^3$ is an open question.  By Theorem \ref{t:main}, $C$ has many chaotic volume-preserving holomorphic automorphisms.

We do not know whether the assumption on the volume form in Theorem \ref{t:main} can be removed.  In particular, it is an open question whether chaos is generic for volume-preserving holomorphic automorphisms of reductive groups.  If chaotic automorphisms are not dense, it will not be for purely topological reasons.  Namely, let $X$ be a Stein manifold of dimension at least $2$ with a holomorphic volume form $\omega$.  Recall that $X$ has only one end.  By a theorem of Alpern and Prasad \cite[Theorem 3]{AP2000}, chaotic homeomorphisms are dense with respect to the compact-open topology among homeomorphisms of $X$ that preserve the real volume form $\omega\wedge\bar\omega$.

\subsection{}
The remainder of the paper is organised as follows.  In Section \ref{sec:perturbations}, guided by Varolin's proof of his \cite[Lemma 3.2]{Varolin2000}, we prove a preparatory theorem (Theorem \ref{t:perturbations}) on small volume-preserving perturbations of the identity with prescribed jets at finitely many points.

In Section \ref{sec:expelling}, we generalise Forn\ae ss and Sibony's proof of their \cite[Theorem 3.1]{FS1997} to the setting of our Theorem \ref{t:main} and establish the following result.

\begin{theorem}  \label{t:expelling}
Let $X$ be a Stein manifold satisfying the volume density property with respect to an exact holomorphic volume form.  A generic volume-preserving holomorphic automorphism of $X$ is expelling, meaning that the $F_\sigma$ set of points with relatively compact forward orbit is nowhere dense.
\end{theorem}

We do not know whether the assumption that the volume form is exact can be removed.  Let us call an automorphism $f$ in $\Aut_\omega(X)$ \textit{robustly non-expelling}\footnote{We thank an anonymous referee for suggesting this term.} if there is a neighbourhood $W$ of $f$ in $\Aut_\omega(X)$, a nonempty open subset $V$ of $X$, and a compact subset $K$ of $X$ such that $g^j(V)\subset K$ for all $g\in W$ and $j\geq 0$.  The proof of Theorem \ref{t:expelling} shows that if expelling automorphisms are not generic in $\Aut_\omega(X)$, then $X$ has a robustly non-expelling automorphism (the converse is obvious).  In Section \ref{sec:robustly-non-expelling}, we present some curious consequences of the existence of such an automorphism.  We hope that these results may help construct examples or prove non-existence.  By Theorem \ref{t:strange-orbit}, a robustly non-expelling automorphism must possess a very peculiar orbit.  By Theorem \ref{t:cohomological-condition}, under a cohomological assumption that is satisfied for example by reductive non-semisimple groups such as $(\C^*)^n$, if expelling automorphisms are not generic, then there is a nonempty open subset $U$ of $\Aut_\omega(X)$ such that for every $f\in U$, there is $s\geq 2$ such that the $s$-th power map $\Aut_\omega(X) \to \Aut_\omega(X)$ is not open at $f$.  By contrast, for a Lie group, each power map is open outside a proper subvariety, so at a generic point they are all open.

In Section \ref{sec:proof}, we use \cite[Theorem 2]{Varolin2000} to reduce Theorem \ref{t:main} to Theorem \ref{t:expelling}.  Finally, in Section \ref{sec:dense-stable-manifolds}, we show how to generalise Peters, Vivas, and Wold's \cite[Theorem 1.1]{PVW2008} to Theorem \ref{t:dense-stable-manifold} using Theorem \ref{t:expelling}.

\section{Perturbations of the identity}
\label{sec:perturbations}

\noindent
Let $X$ be a complex manifold, $\mathscr X_\O(X)$ be the Lie algebra of holomorphic vector fields on $X$, and $J^k(X)$ be the complex manifold of $k$-jets of holomorphic maps from open subsets of $X$ into $X$.  We denote the $k$-jet of a map $f$ at $x\in X$ by $j_x^k(f)$.  Each $v\in\mathscr X_\O(X)$ with flow $\varphi_v^t$ induces $p_k(v)\in \mathscr X_\O(J^k(X))$ with flow
\[ \varphi_{p_k(v)}^t(j_x^k(f)) = j_x^k(\varphi_v^t\circ f). \] 
The map $p_k$ sends complete vector fields to complete vector fields and defines a continuous monomorphism of Lie algebras $p_k:\mathscr X_\O(X) \to \mathscr X_\O(J^k(X))$.

If $\mathfrak g$ is a Lie subalgebra of $\mathscr X_\O(X)$, then $p_k(\mathfrak g)$ is a Lie subalgebra of $\mathscr X_\O(J^k(X))$.  If $\mathfrak g$ has the density property, meaning that the Lie subalgebra of $\mathfrak g$, or equivalently the vector subspace of $\mathfrak g$, generated by the complete fields in $\mathfrak g$ is dense in $\mathfrak g$, then $p_k(\mathfrak{g})$ has the density property as well.

A holomorphic automorphism $\psi$ of $X$ induces a holomorphic automorphism $\psi_\#$ of $J^k(X)$ in the obvious way: $\psi_\#j_x^k(f)=j_x^k(\psi\circ f)$.  If $v$ is a complete holomorphic vector field on $X$, then
\[ (\varphi_v^t)_\#=\varphi_{p_k(v)}^t. \]
Denote by $\Aut_{\mathfrak g}(X)$ the subgroup of $\Aut(X)$ generated by the time maps of complete vector fields in $\mathfrak g$.  Then 
\[ \Aut_{\mathfrak g}(X)_\#=\Aut_{p_k(\mathfrak{g})}(J^k(X)). \]

The orbit of $\mathfrak g$ through $p\in X$, denoted $\mathscr{R}_\mathfrak{g}(p)$, consists of all points in $X$ of the form $\varphi_{v_m}^1\circ \cdots\circ \varphi_{v_1}^1(p)$, where $v_1, \dots , v_m \in \mathfrak{g}$ and the expression makes sense.  By the orbit theorem \cite[Theorem 1, p.~33]{Jurdjevic1997}, $\mathscr{R}_\mathfrak{g}(p)$ is an immersed complex submanifold of $X$.

For the remainder of the section, let $X$ be an $n$-dimensional Stein manifold, $n\geq 2$, with a holomorphic volume form $\omega$, and let $\mathfrak{g}$ be the Lie algebra of holomorphic vector fields on $X$ of $\omega$-divergence zero.  Assume that $X$ has the volume density property with respect to $\omega$, that is, that $\mathfrak g$ has the density property.

The next result is essentially \cite[Lemma 3.2]{Varolin2000}.

\begin{proposition}
Let $x_0\in X$ and $U$ be a neighbourhood of $\id_X$ in $\Aut_\omega(X)$.  There is a neighbourhood $V$ of $j^k_{x_0}(\id_X)$ in $\mathscr{R}_{p_k({\mathfrak{g}})}(j^k_{x_0}(\id_X))$ such that for all $\gamma \in V$, there is $h_\gamma\in U$ with $j^k_{x_0}(h_\gamma)=\gamma$.
\end{proposition}

\begin{proof}
By the Hermann-Nagano theorem \cite[Theorem 6, p.~48]{Jurdjevic1997}, there are $v_1,\ldots,v_m\in\mathfrak g$ such that $p_k(v_1), \ldots, p_k(v_m)\in p_k({\mathfrak{g}})$ form a basis for the tangent space of $\mathscr{R}_{p_k({\mathfrak{g}})}(b)$ at $b=j^k_{x_0}(\id_X)$.  Since $\mathfrak{g}$ has the density property, we may assume that each $v_j$ is the sum of complete vector fields in $\mathfrak{g}$.  Hence we may assume that $v_1,\ldots,v_m$ are complete.

Define a continuous map 
\[ \Phi : \C^m\to \Aut_{p_k(\mathfrak{g})}(J^k(X)), \quad (t_1,\ldots, t_m)\mapsto \varphi^{t_m}_{p_k(v_m)}\circ \cdots \circ \varphi^{t_1}_{p_k(v_1)}. \]
The holomorphic map $\C^m\to\mathscr{R}_{p_k({\mathfrak{g}})}(b)$, $t\mapsto \Phi(t)(b)$, is a submersion at $b$, so it admits a holomorphic section $\sigma$ on a neighbourhood $W$ of $b$ with $\sigma(b)=0$.  The continuous map $\Psi=\Phi\circ \sigma$ takes values in $\Aut_{p_k(\mathfrak{g})}(J^k(X))$ with $\Psi(b)=\id$ and $\Psi(\gamma)(b)=\gamma$ for all $\gamma\in W$.  Let $V$ be a neighbourhood of $b$ such that $\Psi (V)\subset U$, and for each $\gamma\in V$, take $h_\gamma\in \Aut_{\mathfrak{g}}(X)$ such that $(h_\gamma)_\#=\Psi(\gamma)$. 
\end{proof}

We need to be able to work with a finite number of points simultaneously.

Let $Y_{X,m}$ denote the configuration space of $m$-tuples of distinct points in $X$.
 If $v$ is a vector field on $X$, then $\oplus v$ denotes the vector field on $Y_{X,m}$ with $\oplus v(x_1,\ldots,x_m)=(v(x_1),\ldots,v(x_m))$, where we have identified $T_{(x_1,\ldots,x_m)}Y_{X,m}$ with $T_{x_1}X \oplus\cdots\oplus T_{x_m}X$.  Note that $\oplus\mathfrak{g}$ is a Lie algebra with the density property on $Y_{X,m}$.

The same argument as above gives the following result.

\begin{proposition}  \label{p:several-points}
Let $x_1, \ldots, x_m$ be distinct points in $X$.  Let $U$ be a neighbourhood of $\id_X$ in $\Aut_\omega(X)$.  Denote the point $(j^k_{x_1}(\id_X), \dots,j^k_{x_m}(\id_X))$ in $Y_{J^k(X),m}$ by $b$.  There is a neighbourhood $V$ of $b$ in $\mathscr{R}_{\oplus p_k({\mathfrak{g}})}(b)$ such that for all $\gamma=(\gamma_1,\dots, \gamma_m) \in V$, there is $h_\gamma\in U$ with $j^k_{x_j}(h_\gamma)=\gamma_j$ for $j=1,\ldots, m$.
\end{proposition}

Note that if $v$ is a holomorphic vector field with a zero at $p\in X$, then $d_p v$ is a well-defined $\C$-linear endomorphism of $T_{p}X$. 

\begin{proposition}   \label{p:orbit}
Let $x_1,\ldots, x_m$ be distinct points in $X$.  For $j=1,\ldots,m$, let $A_j$ be a $\C$-linear endomorphism of $T_{x_j}X$ with zero trace.  There exists a divergence-free holomorphic vector field $v$ on $X$ such that $v(x_j)=0$ and $d_{x_j}v=A_j$ for $j=1,\ldots,m$.
\end{proposition}

\begin{proof}
Let $\varphi_j$ be a holomorphic coordinate on $X$ near $x_j$ with $\varphi_j(x_j)=0$ and $\omega=\varphi_j^*(dz_1\wedge \cdots \wedge dz_n)$.  Let 
\[ B_j = d_{x_j}\varphi_j\circ A_j\circ (d_{x_j}\varphi_j)^{-1}\in \mathfrak{sl}_n(\C). \]
Take mutually disjoint open balls centred at $x_1,\ldots, x_m$.  Define a holomorphic vector field $\chi$ on their union $\Omega$ to be constantly equal to $\varphi_j^*B_j$ near $x_j$.  Then $\chi$ is divergence-free and $d_{x_j}\chi=A_j$.

Consider the closed holomorphic $(n-1)$-form $\alpha=\chi\rfloor \omega$ on $\Omega$.  Since $H^{n-1}(\Omega)=0$, there is a holomorphic $(n-2)$-form $\beta$ on $\Omega$ such that $d\beta=\alpha$.  Find a holomorphic $(n-2)$-form $\tilde \beta$ on $\C^n$ such that $j^2_{x_j}\tilde\beta=j^2_{x_j}\beta$ for $j=1,\ldots,m$.

Let $\tilde \alpha=d\tilde\beta$ and define a divergence-free holomorphic vector field $v$ on $X$ by $\tilde\alpha=v\rfloor \omega$.  Then $j^1_{x_j}\tilde \alpha=j^1_{x_j}\alpha$, so $j^1_{x_j}v=j^1_{x_j} \chi$ for $j=1,\ldots,m$.
\end{proof}

Here, finally, is the main result of this section.

\begin{theorem}   \label{t:perturbations}
 let $X$ be a Stein manifold of dimension $n\geq 2$, satisfying the volume density property with respect to a holomorphic volume form $\omega$.  Let $x_1,\dots, x_m$ be distinct points in $X$.  Let $u\in T_{x_1}X$, $u\neq 0$.  Every neighbourhood of $\id_X$ in $\Aut_\omega(X)$ contains an automorphism $h$ such that:
\begin{enumerate}
\item $h(x_j)=x_j$ for $j=1,\ldots,m$,
\item $d_{x_j}h=\id$ for $j=2,\ldots,m$,
\item $d_{x_1}h(u)=\alpha u$, with $\alpha>1$.
\end{enumerate}
\end{theorem}

\begin{proof}
Choose an isomorphism of $T_{x_1}X$ with $\C^n$ taking $u$ to $(1,0,\ldots,0)$.  Let $A_1\in \mathfrak{sl}_n(\C)$ be the diagonal matrix with diagonal entries $1,-1,0,\ldots, 0$ and let $A_j=0$ for all $j=2,\ldots,m$.  Let $v$ be the vector field given by Proposition \ref{p:orbit} and let $\varphi^t_v$ be its flow.   We have a continuous map
\[ \gamma:[0,\infty)\to \mathcal{R}_{\oplus p_1({\mathfrak{g}})}(j^1_{x_1}(\id_X), \dots,j^1_{x_m}(\id_X)), \quad \gamma(t) = (j^1_{x_1}(\varphi^t_v), \ldots, j^1_{x_m}(\varphi^t_v)), \]
such that $\gamma(0)=(j^1_{x_1}(\id_X), \ldots, j^1_{x_m}(\id_X))$ and $\gamma(t)$ satisfies (1), (2), (3) for all $t>0$.  Proposition \ref{p:several-points} completes the proof.
\end{proof}

\section{Genericity of expelling automorphisms}
\label{sec:expelling}

\noindent
This section contains the proof of Theorem \ref{t:expelling}.  We let $X$ be a Stein manifold satisfying the volume density property with respect to an exact holomorphic volume form $\omega$.  We will show that a generic volume-preserving holomorphic automorphism of $X$ is expelling.  This is obvious for the sole $1$-dimensional example of $X=\C$ with $\omega=dz$, so we will assume that $n=\dim X\geq 2$.

For $f\in\Aut_\omega(X)$, let $K_f$ be the set of $x\in X$ such that the forward orbit $\{f^j(x):j\geq 0\}$ is relatively compact in $X$.  Note that $K_f$ is $F_\sigma$.  Indeed, if $(\Lambda_k)_{k\geq 0}$ is an exhaustion of $X$ by compact subsets, then for each $k\geq 0$, the set $K_f^k$ of $x\in X$ such that $f^j(x)\in \Lambda_k$ for all $j\geq 0$ is a compact subset of $X$, and clearly, $K_f=\bigcup\limits_{k\geq 0}K_f^k$.  Our goal, therefore, is to show that for a generic $f\in\Aut_\omega(X)$, the $G_\delta$ set $X\setminus K_f$ is dense in $X$.

\begin{claim}   \label{c:Baire-arguments}
It suffices to show that 
\[ K= \bigcup\limits_{f\in\Aut_\omega(X)} K_f\times\{f\} \]
has empty interior in $X\times \Aut_\omega(X)$.
\end{claim}

\begin{proof}
Assume that $K$ has empty interior.  Let $B\Subset X$ be open, let $k\geq 0$, and consider the closed  subset of $\Aut_\omega(X)$ with empty interior defined by
\[ E_B^k = \{f\in \Aut_\omega(X): f^j(\overline B)\subset \Lambda_k \textrm{ for all } j\geq 0\}. \]
Let $\mathscr B$ be a countable basis for the topology on $X$, consisting of relatively compact open subsets, and consider the union 
\[ F=\bigcup_{B\in\mathscr B, k\geq 0}E_B^k. \]
Since $\Aut_\omega(X)$ is a Baire space, $F$ is an $F_\sigma$ subset of ${\rm Aut}_\omega(X)$ with empty interior.  Its complement is a dense $G_\delta$, and for every $f$ in the complement, by Baire, the set $K_f$ has empty interior.
\end{proof}

We assume that the interior of $K$ is not empty.  We will eventually arrive at a contradiction.  If the closed set 
\[ K_k = \{ (x,f)\in X\times \Aut_\omega(X) : f^j(x)\in \Lambda_k \textrm{ for all } j\geq 0\} \]
had empty interior for all $k\geq 0$, then $K=\bigcup\limits_{k\geq 0}K_k$ would have empty interior too, since $X\times \Aut_\omega(X)$ is a Baire space.  Hence $U_k=K_k^\circ\neq\varnothing$ for some $k\geq 0$.  Choose $(x_0, f_0)\in U_k$ and let
\[ A_{f_0} = \{x\in X : (x, f_0)\in U_k\}. \] 
Clearly, $A_{f_0}$ is open, relatively compact, and forward invariant by $f_0$, that is, $f_0(A_{f_0}) \subset A_{f_0}$.

\begin{claim}  \label{c:completely-invariant}
$A_{f_0}$ is completely invariant by $f_0$, that is, $f_0(A_{f_0}) = A_{f_0}$.
\end{claim}

\begin{proof}
Since $f_0$ is volume-preserving, every connected component $\Omega_0$ of $A_{f_0}$ is forward invariant by some iterate $f_0^j$.  Assume that $j=1$ (otherwise consider $f_0^j$ instead of $f_0$).  We will show that $f_0(\Omega_0)=\Omega_0$.  There is a subsequence $(f_0^{n_k})$ converging uniformly on compact subsets of $\Omega_0$ to a holomorphic map $g:\Omega_0\to X$.  Since $f_0$ is volume-preserving, $g$ is injective by Hurwitz' theorem.  Let $x\in \Omega_0$ and let $V\Subset g(\Omega_0)$ be a neighbourhood of $g(x)$.  By \cite[Theorem 5.2]{DE1986}, $V$ is eventually contained in $f_0^{n_k}(\Omega_0)$, so $V\subset \Omega_0$.  Moreover, $(f_0^{-n_k})$ converges to $g^{-1}\vert_V$.  By passing to a subsequence, we may assume that $m_k=n_{k+1}-n_k$ strictly increases to $\infty$, so, on $V$, $f_0^{m_k}$ converges to $\id_V$.  By Vitali's theorem, this implies that $f_0^{m_k} $ converges to $\id_\Omega$.  Again by \cite[Theorem 5.2]{DE1986}, any compact set in $\Omega_0$ is eventually contained in $f_0^{m_k}(\Omega_0)$, that is, $f_0(\Omega_0)=\Omega_0$.
\end{proof}

\begin{claim}  \label{c:noperiodic}
The point $x_0$ is not periodic for $f_0$.
\end{claim}

\begin{proof}
Suppose that $x_0$ is periodic for $f_0$ with period $m$.  Assume that the differential $d_{x_0} f_0^m$ admits an eigenvalue with absolute value $\lambda$ strictly bigger than $1$.  Then for each $j\geq 1$, the map $f_0^{mj}: A_{f_0}\to A_{f_0}$ admits an eigenvalue with absolute value $\lambda^j$.  Let $\gamma$ be a holomorphic disc in  $A_{f_0}$ tangent to an associated eigenvector. Then the family of holomorphic discs $f_0^{mj}\circ \gamma : \mathbb D \to A_{f_0} \Subset X$, $j\geq 1$, contradicts Cauchy estimates at $x_0$.
 
Since $f_0$ preserves the holomorphic volume form $\omega$, its holomorphic Jacobian determinant is $1$.  Hence, if $d_{x_0} f_0^m$ has no eigenvalue with absolute value strictly bigger than $1$, then all the eigenvalues of $d_{x_0} f_0^m$ have absolute value $1$.  Assume this.  Let $v$ be an eigenvector for $d_{x_0} f_0^m$.  Let $x_j=f_0^j(x_0)$ for $j\geq 0$.  By Theorem \ref{t:perturbations}, there is $h\in \Aut_\omega(X)$ such that 
\begin{enumerate}
\item $h$ is arbitrarily close to $\id_X$,
\item $h(x_j)=x_j$ for $j=0,\ldots,m-1$.
\item $d_{x_0}h(v)=\alpha v$, with $\alpha>1$,
\item $d_{x_j}h=\id$ for $j=1,\ldots,m-1$.
\end{enumerate}
Let $f_1=h\circ f_0$.  Then $x_0$ is a periodic point of period $m$ for $f_1$ and $v$ is an eigenvector of $d_{x_0}f_1^m$ whose eigenvalue has absolute value strictly greater than $1$.  If $h$ is close enough to $\id_X$, then $f_1$ is close enough to $f_0$ that the point $(x_0, f_1)$ belongs to $U_k$.  We obtain a contradiction as before.
\end{proof}

Every connected component of $A_{f_0}$ is periodic.  Let $\Omega_0$ be a connected component and let $\Omega$ be the union of all the connected components in the $f_0$-cycle of $\Omega_0$.  Then $f_0\in\Aut(\Omega)$.  Since $\Omega$ is a relatively compact open set in a Stein manifold, it is Kobayashi-hyperbolic, so $\Aut (\Omega)$ is a real Lie group.
 
\begin{claim}
The closure $G$ in $\Aut(\Omega)$ of the group generated by $f_0$ is a compact abelian Lie subgroup of $\Aut(\Omega)$.
\end{claim}

\begin{proof}
It is clear that $G$ is an abelian Lie subgroup of the closed subgroup $\Aut_\omega(\Omega)$ of $\Aut(\Omega)$.  Since $\Omega$ is relatively compact in the Stein manifold $X$, it follows that $\O(\Omega,\Omega)$ is relatively compact in $\O(\Omega,X)$.  Hence, to show that $G$ is compact, it suffices to show that $\Aut_\omega(\Omega)$ is closed in $\O(\Omega,X)$.  Let $(g_j)$ be a sequence in $\Aut_\omega(\Omega)$ converging uniformly on compact subsets of $\Omega$ to a holomorphic map $g:\Omega\to X$.  Arguing as in the proof of Claim \ref{c:completely-invariant}, we find that $g$ is injective, volume-preserving, and $g(\Omega)\subset \Omega$.  Now consider the sequence of inverses $(g_j^{-1})$ in $\Aut_\omega (\Omega)$.  Montel's theorem yields a subsequence converging to a holomorphic map $h:\Omega\to X$.  Arguing as before, we obtain that $h(\Omega)\subset \Omega$, so $g(\Omega)=\Omega$.
\end{proof}

Every compact abelian Lie group is isomorphic to $\mathbb T^\ell\times A$, where $A$ is a finite abelian group, $\ell\geq 0$, and $\mathbb T$ is the circle.  If $\ell=0$, then $G=A$ is a finite abelian group, and it follows that some iterate of $f_0$ is $\id_\Omega$, contradicting Claim \ref{c:noperiodic}.  Hence, $\ell\geq 1$.  The action of $G$ on $\Omega$ is continuous and hence real-analytic \cite[Satz 6]{Kaup1965}.
The $G$-orbit $Gq$ of each $q\in\Omega_0$ is a compact real-analytic submanifold of $X$ (without boundary and not necessarily connected).

\begin{claim}  \label{c:totallyreal}
Every orbit $Gq$, $q\in\Omega_0$, is totally real.
\end{claim}

The claim is an immediate consequence of the following result.

\begin{proposition}
Let a torus $\mathbb T^\ell$ act by biholomorphisms on an open subset $V$ of a Stein manifold $M$.  Then every orbit is totally real.
\end{proposition}

\begin{proof}
Let $E$ be the holomorphic envelope of $V$.  It is a Stein Riemann domain $\pi:E\to M$ with a holomorphic embedding $\iota:V\to E$ such that $\pi\circ\iota$ is the inclusion $V\hookrightarrow M$.  The action can be extended to a $\mathbb T^\ell$-action by biholomorphisms on $E$.  Hence we can assume that $V$ is Stein.

By Heinzner's complexification theorem \cite{Heinzner1991}, there exists a Stein manifold $\tilde V$ containing $V$ as an open subset, such that the action of $\mathbb T^\ell$ on $V$ extends to an action of the complexification $(\C^*)^\ell$ on $\tilde V$.  Let $I_q$ be the $(\C^*)^\ell$-stabiliser of $q\in V$.  The action induces an injective holomorphic immersion $\Psi_q : (\C^*)^\ell/I^q\to \tilde V$, whose image is the $(\C^*)^\ell$-orbit of $q$.  Let $\pi_q : (\C^*)^\ell\to (\C^*)^\ell/I^q$ be the quotient map and consider the compact real Lie subgroup $\pi_q(\mathbb T^\ell)$ of $(\C^*)^\ell/I^q$.  We will show that $\pi_q(\mathbb T^\ell)$ is totally real in $(\C^*)^\ell/I^q$; it follows that the $\mathbb T^\ell$-orbit $\Psi_q(\pi_q(\mathbb T^\ell))$ of $q$ is totally real.  

By the classification of abelian complex Lie groups \cite{Morimoto1966}, $(\C^*)^\ell/I^q$ is isomorphic to a product $H\times \C^r\times (\C^*)^s$, where $H$ is a complex Lie group with no nonconstant holomorphic functions.  Since $\tilde V$ is Stein, $H$ is trivial, so $\pi_q(\mathbb T^\ell)$ is contained in the totally real maximal torus in $(\C^*)^s$.
\end{proof}

Next we show that there is $q\in\Omega_0$ whose orbit $Gq$ is $\O(X)$-convex.  The special orbit $Gq$ will then have a basis of open neighbourhoods that are Stein and Runge.  We denote by $\widehat K$ the $\O(X)$-hull of a compact subset $K$ of $X$.

\begin{lemma}   \label{l:difficult}
Let $q_0, q_1\in\Omega_0$.  If $q_1\in \widehat {Gq_0}\setminus G q_0$, then $\widehat{Gq_1}\subset \widehat {Gq_0}\setminus Gq_0$.
\end{lemma}

\begin{proof}
From the group action, it is clear that $Gq_1\subset \widehat{Gq_0}$, so $\widehat{Gq_1}\subset \widehat{Gq_0}$, and $Gq_1\cap Gq_0=\varnothing$.  We need to show that $\widehat{Gq_1}\cap Gq_0=\varnothing$.

Let the uniform algebra $A$ consist of those continuous functions on the compact set $\widehat{Gq_0}$ that are uniform limits of holomorphic functions on $X$.  It is well known that $A$ has a peak point in $Gq_0$.  Using the group action, we see that every point in $Gq_0$ is a peak point for $A$.

Now let $p\in Gq_0$.  Since $p$ is a peak point for $A$, there is $f\in A$ with $\lvert f(p)\rvert=1$ but $\lvert f\rvert<1$ on $\widehat{Gq_0}\setminus\{p\}\supset Gq_1$, so $p\notin\widehat{Gq_1}$.
\end{proof}

\begin{claim}
There is $q\in \Omega_0$ whose orbit $Gq$ is $\O(X)$-convex.
\end{claim}

\begin{proof}
We define a partial order on the set $\{\widehat{Gp}: p\in \Omega_0\}$ by reverse inclusion:  $\widehat{Gq}\leq\widehat{Gp}$ if and only if $\widehat{Gp}\subset\widehat{Gq}$.

For $\widehat{Gq}$ to be maximal means that if $\widehat{Gp}\subset \widehat{Gq}$, then $\widehat{Gp}=\widehat{Gq}$.  It follows that $\widehat{Gq}={Gq}$.  Indeed, if there is a point $p\in \widehat{Gq}\setminus Gq$, then by Lemma \ref{l:difficult}, $\widehat{Gp}\subsetneqq \widehat{Gq}$.

Let $\mathscr C$ be a totally ordered subset of $\{\widehat{Gq}: q\in \Omega_0\}$.  The intersection $\bigcap\mathscr C$ is not empty by compactness.  Take $q\in \bigcap\mathscr C$. Then $\widehat{Gq}$ is an upper bound for $\mathscr C$.  Zorn's lemma now provides a maximal element.
\end{proof}

From now on we fix $q\in\Omega_0$ such that $Gq$ is $\O(X)$-convex.  Since $Gq$ is totally real, $\dim Gq\leq n$.  By Claim \ref{c:noperiodic}, $\dim Gq\geq 1$.  Next we use the assumption that $\omega$ is exact to rule out $\dim Gq=n$.  This is the only place in the proof of Theorem \ref{t:expelling} where exactness of $\omega$ is used.

\begin{claim}  \label{c:dim-n-ruled-out}
$\dim Gq<n$.
\end{claim}

\begin{proof}
Suppose that $\dim Gq=n$.  Then $Gq$ may be identified with a Lie group of the form $\mathbb T^n\times A$, where $A$ is a finite abelian group, embedded as a totally real real-analytic submanifold of $X$.  In suitable local holomorphic coordinates $z_1,\ldots,z_n$ on $X$ at each of its points, $Gq$ is defined by the equations $\Im z_1=0,\ldots,\Im z_n=0$, so $\omega\vert_{Gq}$ has no zeros.  On the other hand, since $\omega\vert_{Gq}$ is exact and invariant under the action of $f_0$ and hence the action of $G$ and hence the action of $Gq$ on itself, we conclude that $\omega\vert_{Gq}=0$.
\end{proof}

We have established that $1\leq \dim Gq\leq n-1$.  Take $s\geq 1$ such that $f_0^s$ is in the identity component $G_0$ of $G$.  There is a 1-parameter subgroup $(g_t)_{t\in\mathbb R}$ of $G_0$ such that $g_1=f_0^s$.  Consider the vector field $\xi=\dfrac{d}{dt}g_t\bigg\vert_{t=0}$ on $\Omega$.  It is holomorphic, divergence-free, tangent to the $G$-orbits in $\Omega$, and it does not have any zeros, for if it did, $f_0$ would have a periodic point in $\Omega$, contradicting Claim \ref{c:noperiodic}.

\begin{claim}  \label{c:approx}
There is a Stein and Runge neighbourhood $U$ of $Gq$ such that $\xi$ is uniformly approximable on compact subsets of $U$ by divergence-free holomorphic vector fields on~$X$.  
\end{claim}

\begin{proof}
Recall the bijective correspondence between divergence-free holomorphic vector fields $\xi$ and closed holomorphic $(n-1)$-forms $\beta$ given by $\beta=\xi\rfloor\omega$.  On a Stein and Runge open subset $U$ of $\Omega$, the vector field $\xi$ is approximable by divergence-free holomorphic vector fields on $X$ if and only if $\xi\rfloor\omega$ is approximable by closed holomorphic $(n-1)$-forms on $X$, which is the case if the cohomology class of $\xi\rfloor\omega$ in $H^{n-1}(U)$ lies in the image of $H^{n-1}(X)$.

Let $W$ be a tubular neighbourhood of $Gq$.  It suffices to show that the cohomology class of $\xi\rfloor\omega$ in $H^{n-1}(W)$ is zero.  In fact, the restriction of $\xi\rfloor\omega$ to $Gq$ is zero as a form and not only as a cohomology class.  This is clear if $\dim Gq\leq n-2$.  If $\dim Gq=n-1$ and we take vectors $v_1,\ldots,v_{n-1}$ in the tangent space of $Gq$ at $x\in Gq$, then the vectors $\xi(x), v_1,\ldots,v_{n-1}$ are linearly dependent because $\xi$ is tangent to $Gq$, so $\xi\rfloor\omega(v_1,\ldots,v_{n-1}) =\omega(\xi(x),v_1,\ldots,v_{n-1})=0$.
\end{proof}

By Claim \ref{c:approx}, we can approximate $\xi$ by a divergence-free holomorphic vector field $\tilde \eta$ on $X$ uniformly on $Gq$.  Since $X$ has the volume density property, $\tilde \eta$ can be approximated uniformly on $Gq$ by a vector field $\eta$ which is the sum of complete divergence-free holomorphic vector fields $v_1, \dots, v_m$ on $X$.  Let $\varphi^j_t$ be the flow of $v_j$.  Let $\eta$ approximate $\xi$ so well that $\eta$ has no zeros in $Gq$.

For $t\in \C$, let
\[ \Psi_t=\varphi^m_t\circ \cdots\circ \varphi^1_t \in\Aut_\omega(X). \] 
Note that 
\[ \frac{\partial}{\partial t}\Psi_t(x)\bigg\vert_{t=0}=\eta(x) \quad\textrm{for all }x\in X. \]
Recall our assumption that $(x_0,f_0)\in U_k$, so $\Omega\times\{f_0\}\subset U_k$.
 Since $U_k$ is open and $Gq$ is compact, there is a neighbourhood $W$ of $f_0$ in $\Aut_\omega(X)$ such that $Gq\times W\subset U_k$.

Since $\Psi_t\to \id_X$ as $t\to 0$, there is $\delta>0$ such that $\Psi_t\circ f_0\in W$ when $\lvert t\rvert<\delta$.  Hence 
\[ (\Psi_t\circ f_0)^j(x)\in \Lambda_k \quad \textrm{when } \lvert t\rvert<\delta, \  x\in Gq, \  j\geq 0. \]
Let $\lVert\cdot\rVert$ be a hermitian metric on $X$.  Cauchy estimates show that there is a constant $M$ such that for all $x\in Gq$ and $j\geq 0$,
\[ \bigg\lVert\frac{\partial}{\partial t}(\Psi_t\circ f_0)^j(x)\bigg\vert_{t=0}\bigg\rVert\leq M. \]

To elucidate how we will obtain a contradiction, let us pretend that $\Psi_t$ commutes with $f_0$.  Then 
\[ \frac{\partial}{\partial t}(\Psi_t\circ f_0)^j(x)\bigg\vert_{t=0}=j\frac{\partial}{\partial t}\Psi_t(f_0^j(x))\bigg\vert_{t=0}, \]
but since $\dfrac{\partial}{\partial t}\Psi_t(x)\bigg\vert_{t=0}\neq 0$ for all $x\in Gq$, we have
\[ \min_{y\in Gq} \bigg\lVert\frac{\partial}{\partial t}\Psi_t(y)\bigg\vert_{t=0}\bigg\rVert>0, \]
contradicting the estimate above.

In reality, we  use the fact that $f_0$ and $\xi$ commute, that is,
\[ d_x f_0(\xi (x)) = \xi(f_0(x)) \quad \textrm{for all } x\in\Omega. \]
By the chain rule, the derivative $\dfrac\partial{\partial t}(\Psi_t\circ f_0)^j(x)\bigg\vert_{t=0}$ is
\[ \frac{\partial}{\partial t}\Psi_t(f_0^j(x))\bigg\vert_{t=0}+d_{f_0^{j-1}(x)}f_0\big(\frac{\partial}{\partial t}\Psi_t(f_0^{j-1}(x))\bigg\vert_{t=0}\big)+d_{f_0^{j-2}(x)}f_0^2\big(\frac{\partial}{\partial t}\Psi_t(f_0^{j-2}(x))\bigg\vert_{t=0}\big)+\cdots, \]
that is,
\[ \eta(f_0^j(x))+d_{f_0^{j-1}(x)}f_0(\eta (f_0^{j-1}(x)))+d_{f_0^{j-2}(x)}f_0^2(\eta (f_0^{j-2}(x)))+\cdots. \]
Let 
\[ M=\sup\limits_{x\in Gq, \, j\geq 0}\lVert d_x f_0^j\rVert, \] 
which is finite since the images of the maps $f_0^j$ near $Gq$ are contained in $\Lambda_k \Subset X$.  Let 
\[ c=\min_{x\in Gq} \lVert \xi(x)\rVert >0. \]
Assume that $\lVert\xi-\eta\rVert\leq \dfrac{c}{2M}$ on $Gq$. Then for $x\in Gq$ and $i=0,\ldots,j-1$,
\[ \lVert d_{f_0^{j-i}(x)}f_0^i\big(\eta (f_0^{j-i}(x))\big)-\xi (f_0^j(x))\rVert=\lVert d_{f_0^{j-i}(x)}f_0^i\big(\eta (f_0^{j-i}(x))-\xi(f_0^{j-i}(x))\big)\rVert \leq\frac c 2. \]
Again we obtain a contradiction.  Namely, for $x\in Gq$, the derivative $\dfrac\partial{\partial t}(\Psi_t\circ f_0)^j(x)\bigg\vert_{t=0}$ is bounded as $j\to\infty$, but is also within $\dfrac c 2 j$ of $j\xi(f_0^j(x))$, whose norm is at least $cj$.

\section{Robustly non-expelling automorphisms}
\label{sec:robustly-non-expelling}

\noindent
Much interesting information can be extracted from the proof of Theorem \ref{t:expelling}.  First, the following result is evident from the proof of Claim \ref{c:Baire-arguments}.

\begin{proposition}   \label{p:G-delta}
Let $X$ be a complex manifold.  The expelling automorphisms of $X$ form a $G_\delta$ subset of $\Aut(X)$.
\end{proposition}

Let $X$ be a Stein manifold of dimension $n\geq 2$ satisfying the volume density property with respect to a holomorphic volume form $\omega$.  When $\omega$ is not exact, it is an open question whether expelling automorphisms are generic in $\Aut_\omega(X)$ or, equivalently by Proposition \ref{p:G-delta}, dense.

\begin{theorem}  \label{t:strange-orbit}
Let $X$ be a Stein manifold of dimension $n\geq 2$ satisfying the volume density property with respect to a holomorphic volume form $\omega$.  Suppose that expelling automorphisms are not generic in $\Aut_\omega(X)$.

{\rm (a)}  Then $X$ has a robustly non-expelling volume-preserving holomorphic automorphism.

{\rm (b)}  Every robustly non-expelling volume-preserving holomorphic automorphism has a total orbit whose closure is the union of finitely many, mutually disjoint, n-dimensional, holomorphically convex, totally real, real-analytic tori in $X$.
\end{theorem}

\begin{proof}
This is immediate from an inspection of the proof of Theorem \ref{t:expelling}.  The only part of the proof that can fail, and that must fail if expelling automorphisms are not generic, is the part where the existence of an orbit $Gq$ of dimension $n$ for the robustly non-expelling automorphism $f_0$ is ruled out.  The proof shows that such an orbit is the closure of a total $f_0$-orbit as described above.
\end{proof}

For the next result, we assume that the class of $\omega$ in $H^n(X)$ lies in the cup product image of $H^1(X)\times H^{n-1}(X)$.  In other words, $\omega$ is of the form $\alpha\wedge\beta+d\gamma$, where $\alpha$ is a closed 1-form on $X$, $\beta$ is a closed $(n-1)$-form, and $\gamma$ is an $(n-1)$-form (these forms may be taken to be holomorphic).  This is obvious if $\omega$ is exact; in particular, if $H^n(X)=0$.  Recall that $X$ has the homotopy type of a smooth manifold $S$ of real dimension at most $n$.  If $S$ is noncompact or $S$ is not orientable or $\dim S<n$, then $H^n(X)\cong H^n(S)=0$.  If $S$ is compact and orientable and $\dim S=n$, then $H^n(X)$ is 1-dimensional, and by Poincar\'e duality, the cup product $H^1(X)\times H^{n-1}(X)\to H^n(X)$ is a nondegenerate pairing.  In particular, the cup product is surjective if $H^1(X)\cong H^{n-1}(X)$ is nontrivial.  Thus the new assumption holds if $H^1(X)\neq 0$ or $H^{n-1}(X)\neq 0$.

If $Y$ is a Stein manifold with the volume density property, then the new assumption holds for $X=Y\times\C^*$.  If a connected linear algebraic group $G$ of dimension $n\geq 2$ is reductive but not semisimple, for example $(\C^*)^n$, then $H^1(G)\cong H^{n-1}(G)\neq 0$, so the new assumption holds for $G$.  If $G$ is semisimple, then $H^1(G)\cong H^{n-1}(G)=0$.  Furthermore, a bi-invariant Haar form on a reductive group is not exact \cite[Proposition 8.6]{KK2016}.  Hence, the new assumption does not capture any semisimple groups.

\begin{theorem}  \label{t:cohomological-condition}
Let $X$ be a Stein manifold of dimension $n\geq 2$ satisfying the volume density property with respect to a holomorphic volume form whose class in $H^n(X)$ lies in the cup product image of $H^1(X)\times H^{n-1}(X)$.  Suppose that expelling automorphisms are not generic in $\Aut_\omega(X)$.

{\rm (a)}  There is a nonempty open subset $U$ of $\Aut_\omega(X)$ such that for every $f\in U$, there is $s\geq 2$ such that the $s$-th power map $\Aut_\omega(X) \to \Aut_\omega(X)$ is not open at $f$.

{\rm (b)}  The closure of the total orbit described in Theorem \ref{t:strange-orbit} is the union of at least two tori.
\end{theorem}

\begin{proof}
By assumption, $X$ has a robustly non-expelling automorphism $f_0$ with an $n$-dimensional orbit $Gq$ as in the proof of Theorem \ref{t:expelling}.  Write $\omega = \alpha\wedge\beta+d\gamma$ as above.  If $\beta$ is exact on $Gq$, then so is $\omega$, which is absurd as shown in the proof of Claim \ref{c:dim-n-ruled-out}.  Hence there is a nonzero class $b$ in the image of $H^{n-1}(X)$ in $H^{n-1}(Gq)$, namely the class of $\beta$.  

Recall the identification of $Gq$ with the Lie group $\mathbb T^n\times A$, where $A$ is a finite abelian group.  Every class in $H^{n-1}(Gq)$ has a unique $\mathbb T^n$-invariant representative.  The assignment $\xi\mapsto [\xi\rfloor\omega]$ is an isomorphism to $H^{n-1}(Gq)$ from the space of $\mathbb T^n$-invariant real-analytic vector fields on $Gq$.  Let $\xi$ have $[\xi\rfloor\omega]=b$.  Then $\xi$ has no zeros and extends to a $G_0$-invariant divergence-free holomorphic vector field on a Stein and Runge neighbourhood $U$ of $Gq$ (as before, $G_0$ denotes the identity component of $G$).  As in the proof of Claim \ref{c:approx}, we can show that $\xi$ is uniformly approximable on compact subsets of $U$ by divergence-free holomorphic vector fields on $X$.

The contradiction at the end of the proof of Theorem \ref{t:expelling} now goes through in either of two circumstances.  First, recall that $s\geq1$ was chosen so that $f_0^s\in G_0$.  If the image of the neighbourhood $W$ of $f_0$ under the $s$-th power map is a neighbourhood of $f_0^s$, then the contradiction goes through with $f_0$ replaced by $f_0^s$.  Also, note that at the beginning of the proof of Theorem \ref{t:expelling}, we made a choice of $(z_0,f_0)$ from the open set $U_k$.  Thus (a) is proved.  Second, if $Gq$ is connected, then $\xi$ is $f_0$-invariant and the contradiction goes through unchanged.  This proves (b).
\end{proof}

\section{Genericity of chaotic automorphisms}
\label{sec:proof}

\noindent
In this section we prove Theorem \ref{t:main}.  We assume that $X$ is a Stein manifold of dimension at least $2$ satisfying the volume density property with respect to an exact holomorphic volume form $\omega$.  We will show that $\Aut_\omega(X)$ contains a dense $G_\delta$ subset consisting of chaotic automorphisms.  Let $U$ and $V$ be nonempty open subsets of $X$ and let
\[ A = \{f\in\Aut_\omega(X) : f \textrm{ has a transverse cycle through } U \textrm{ and } V \}. \]
Recall that a fixed point $p$ of an automorphism $g$ is said to be transverse if the derivative of $g$ at $p$ does not have 1 as an eigenvalue.  A cycle of $g$ of length $k$ is called transverse if the points of the cycle are transverse fixed points of $g^k$.  Since transverse fixed points are stable under perturbation, $A$ is open.  Hence, by Touhey's formulation of chaos, it suffices to show that $A$ is dense in $\Aut_\omega(X)$.

Let $f\in\Aut_\omega(X)$.  We need to show that $f$ can be uniformly approximated on any compact subset $K$ of $X$ by elements of $A$.  By Theorem \ref{t:expelling}, we may assume that both $f$ and $f^{-1}$ are expelling.  We may also assume that $K$ is $\O(X)$-convex.

Choose $p\in U$ such that neither the forward $f$-orbit nor the backward $f$-orbit of $p$ is contained in $f(K)$.  Choose $q\in V$ with the same property.  Then there are $n_1, n_2, m_1, m_2 \geq 1$ such that the points $f^{n_1+1}(p)$, $f^{-n_2}(p)$, $f^{m_1+1}(q)$, $f^{-m_2}(q)$ lie outside $f(K)$.  By \cite[Theorem 2]{Varolin2000}, there is $g\in \Aut_\omega(X)$ such that 
\begin{enumerate}
\item $g$ is as close to $\id_X$ as desired on $f(K)$,
\item $g(f^n(p))= f^n(p)$ and $d_{f^n(p)}g  =\id$ for $-n_2< n\leq n_1$,
\item $g(f^m(q))= f^m(q)$ and $d_{f^m(q)}g  =\id$ for $-m_2< m\leq m_1$,
\item $g(f^{n_1+1}(p))= f^{-m_2}(q)$ and $d_{f^{n_1+1}(p)}g=L_1$, 
\item $g(f^{m_1+1}(q))= f^{-n_2}(p)$ and $d_{f^{m_1+1}(q)}g=L_2,$
\end{enumerate}
where $L_1,L_2$  are linear isomorphisms that preserve $\omega$ such that
\[ L_2 \circ d_{f^{-m_2}(q)}f^{m_1+m_2+1} \circ  L_1\circ d_{f^{-n_2}(p)}f^{n_1+n_2+1} \]
has no eigenvalue equal to 1.  Let $h=g\circ f$.  Then the points
\[ f^{-n_2}(p), \dots, f^{n_1}(p), f^{-m_2}(q),\dots, f^{m_1}(q) \]
form a transverse cycle for $h$ passing through $U$ and $V$.  Thus $h\in A$ and $h$ approximates $f$ as well as desired on $K$.  The proof of Theorem \ref{t:main} is complete.

\section{Genericity of dense stable manifolds}
\label{sec:dense-stable-manifolds}

\noindent
In this section we prove Theorem \ref{t:dense-stable-manifold}.  We continue to assume that $X$ is a Stein manifold of dimension at least $2$ satisfying the volume density property with respect to an exact holomorphic volume form $\omega$.  We show that a generic volume-preserving holomorphic automorphism of $X$ admits a hyperbolic fixed point whose stable manifold is dense in~$X$.

Denote by $\Sigma_f(p)$ the stable manifold of a hyperbolic fixed point $p$ of $f\in\Aut_\omega(X)$.  Preservation of volume implies that $p$ is not repelling, so its attracting dimension is at least 1.  In suitable holomorphic coordinates, $p$ is the origin and $\Sigma_f(p)$ is locally a graph $\Gamma_f(\delta)$ over a small polydisc.  Let $d$ be a distance inducing the topology on $X$, and let $k$ be a distance inducing the topology on $\Aut_\omega(X)$.  Let $\{\varphi_1, \varphi_2, \ldots\}$ be a countable dense subset of $\Aut_\omega(X)$. 

\begin{lemma}
For each $j\geq 1$, there is $\psi_j\in\Aut_\omega(X)$ with a hyperbolic fixed point $p$ such that $k(\psi_j,\varphi_j)<\frac{1}{j}$.
\end{lemma}

\begin{proof}
Let $j\geq 1$.  Let $K\subset X$ be compact and polynomially convex, and $\epsilon>0$ be such that every $g\in\Aut_\omega(X)$ with $d_K(\varphi_j, g)<\epsilon$ also satisfies $k(\varphi_j,g)<\frac{1}{j}$.  Let $p\notin \varphi_j(K)\cup K$.  By \cite[Theorem 2]{Varolin2000}, there is $h\in\Aut_\omega(X)$ such that $d_{\varphi_j(K)}(h,\id)<\epsilon$, $h(\varphi_j(p))=p$, and $d_{\varphi_j(p)}h\circ d_p\varphi_j$ has no eigenvalues of modulus 1.  Now take $\psi_j=h\circ \varphi_j$.
\end{proof}

Denote $p$ by $\eta(\psi_j)$.  Let $m_j\geq 1$ be its attracting dimension.  We can inductively find numbers $\gamma_j>0$ and compact subsets $H_j$ of $X$ such that $\eta(\psi_j)\in H_j^\circ$ and such that the open sets 
\[ B_j=\{h\in \Aut_\omega(X) : d_{H_j}(\psi_j, h)<\gamma_j\} \]
are mutually disjoint.

By \cite[Lemma 1]{PVW2008}, for $\gamma_j$ small enough, there is a neighbourhood $U\Subset H_j^\circ$ of $\eta(\psi_j)$ such that every $h\in B_j$ admits a unique hyperbolic fixed point $\eta(h)$ in $U$ of attracting dimension $m_j$, which varies continuously in the sense that if $(\ell_i)$ is a sequence in $\Aut_\omega(X)$ such that $d_{H_j}(h,\ell_i)\to 0$, then $\eta(\ell_i)\to \eta(h)$.  (Lemma 1 and Corollary 1, cited below, stated and proved in \cite{PVW2008} for $\C^k$, are easily seen to generalise to our $X$.)

For each $h\in B_j$, there is $\delta>0$ such that locally near $\eta(h)$ in suitable holomorphic coordinates, the stable manifold $\Sigma_{h}(\eta(h))$ is a graph $\Gamma_{h}(\delta)$ over a small polydisc
\[ \Delta^{m_j}(\delta)=\{z\in \C^n : z_i=0 \textrm{ for }  i>m_j, \ \lvert z_i\rvert<\delta \textrm{ for } i\leq m_j\}. \]
Moreover, if $(\ell_i)$ is a sequence in $\Aut_\omega(X)$ such that $d_{H_j}(h,\ell_i)\to 0$, then for large $i$ the stable manifold $\Sigma_{\ell_i}(\eta(\ell_i))$ is, locally near $\eta(h)$, a graph $\Gamma_{\ell_i}(\delta)$ over $\Delta^{m_j}(\delta)$ converging to $\Gamma_{h}(\delta)$ in the Hausdorff distance.

Define $A=\bigcup\limits_{j\geq 1} B_j$.  Clearly $A$ is open and dense in $\Aut_\omega(X)$.

\begin{lemma}
Let $q\in X$ and $\epsilon>0$.  The subset
\[ U(q,\epsilon)=\{f\in A : d(\Sigma_f(\eta(f)), q)<\epsilon\} \]
of $\Aut_\omega(X)$ is open and dense.
\end{lemma}
\begin{proof}
By \cite[Corollary 1]{PVW2008}, $U(q,\epsilon)$ is open. We show that it is dense.  Let $f\in \Aut_\omega(X)$.  Let $K\subset X$ be compact and $\delta>0$.  We will show that there is $g\in U(q,\epsilon)$ with $d_K(f,g)\leq \delta$.  Since $A$ is open and dense, and by Theorem \ref{t:expelling}, we can find an expelling automorphism $\tilde f\in\Aut_\omega(X)$ such that $d_K(\tilde f,f)\leq \frac{\delta}{2}$ and such that there is $j\geq 1$ with $\tilde f\in B_j$.  Replace $K$ by a bigger $\mathscr{O}(X)$-convex compact set containing $H_j$.
 
Choose $\tilde q\in B(q,\frac{\epsilon}{2})$ whose forward $\tilde f$-orbit is not relatively compact.  Let $N\geq 1$ be the smallest integer such that $\tilde f^N(\tilde q)\in X\setminus K$.  The stable manifold $\Sigma_{\tilde f}(\eta(\tilde f))$ is not relatively compact, being biholomorphic to $\C^{m_j}$, so there is $x\in \Sigma_{\tilde f}(\eta(\tilde f))\setminus K$ such that 
 $\tilde f^n(x)\in K$ for all $n\geq 1$.
 
Choose holomorphic coordinates near $\eta(\tilde f)$ such that  $\Sigma_{\tilde f}(\eta(\tilde f))$ is locally  a graph $\Gamma_{\tilde f}(\delta)$ as described above.  Take $M\geq 1$ such that $\tilde f^M(x)\in \Gamma_{\tilde f}(\delta)$.  Let $W$ be an $\O(X)$-convex neighbourhood of $K$ containing neither $\tilde f^N(\tilde q)$ nor $x$.  Let $V$ be a neighbourhood of $\tilde f^N(\tilde q)$ and let $\varphi : [0,1]\times V\to X$ be a $C^1$-smooth isotopy such that for all $t\in [0,1]$,
\begin{enumerate}
\item $\varphi_t : V\to X$ is holomorphic, injective, and volume-preserving,
\item $\varphi_t(V)$ is disjoint from $W$, 
\item $W\cup \varphi_t(V)$ is $\O(X)$-convex,
\item $\varphi_0$ is the inclusion of $V$ into $X$, 
\item $\varphi_1(\tilde f^N(\tilde q))=x$.
\end{enumerate}
There are relatively compact neighbourhoods $U$ of $\tilde q$ and $Z$ of $\tilde f^M(x)$ such that 
\begin{enumerate}
\item $U\subset B(q,\epsilon)$,
\item $\tilde f^N(U)\Subset V$,
\item $Z\Subset\tilde f^M\circ \varphi\circ \tilde f^N(U)$.
\end{enumerate}
By the Anders\'en-Lempert theorem for Stein manifolds with the volume density property, there is a sequence $(\Phi_j)$ in $\Aut_\omega(X)$ such that $\Phi_j\to \id$ on $W$ and $\Phi_j\to \varphi_1$ on $V$, uniformly on compact subsets.  For $j$ large enough, the automorphism $g=\tilde f\circ \Phi_j$ satisfies the following conditions:
\begin{enumerate}
\item $g\in B_j$,
\item $d_{K}(\tilde f,g)<\frac{\delta}{2}$, 
\item $Z\Subset g^{M+N+1}(U), $
\item  $ \Gamma_{g}(\delta)\cap Z\neq \varnothing.$
\end{enumerate}
Hence $U$ intersects the stable manifold $\Sigma_{g}(\eta(g))$.
\end{proof}

\begin{proof}[Proof of Theorem \ref{t:dense-stable-manifold}]
Let $\{q_j:j\geq 1\}$ be a dense subset of $X$.  Let $\epsilon_j$, $j\geq 1$, be positive numbers converging to zero.  Since $\Aut_\omega(X)$ has the Baire property, the intersection $\bigcap\limits_{j\geq 1} U(q_j,\epsilon_j)$ is a dense $G_\delta$.  If $f$ is an automorphism in the intersection, then the stable manifold $\Sigma_f(\eta(f))$ is dense in $X$.
\end{proof}

\end{document}